\documentclass[11pt]{article}

\usepackage[a4paper]{anysize}\marginsize{3.5cm}{3.5cm}{1.3cm}{2cm}
\pdfpagewidth=\paperwidth \pdfpageheight=\paperheight
\usepackage{amsfonts,amssymb,amsthm,amsmath,eucal}
\usepackage{pgf}
\usepackage{bbm}
\usepackage{tikz} 
\usepackage{mathrsfs}
\usetikzlibrary{arrows}
\usepackage{color}
\usepackage{float}
\usepackage{subfigure}
\usepackage{caption}
\usepackage{hyperref}

\definecolor{qqqqff}{rgb}{0.,0.,0.}

\theoremstyle{plain}
\newtheorem{thm}{Theorem}[section]
\newtheorem{theorem}[thm]{Theorem}

\newtheorem{lemma}[thm]{Lemma}
\newtheorem{proposition}[thm]{Proposition}

\theoremstyle{definition}
\newtheorem{definition}[thm]{Definition}
\newtheorem{remark}[thm]{Remark}
\newtheorem{example}[thm]{Example}

\newtheorem{problem}[thm]{Problem}

\newtheorem{thevarthm}[thm]{\varthmname}

\newenvironment{varthm*}[1]{\trivlist\item[]{\bf #1.}\it}{\endtrivlist}

\renewcommand\geq{\geqslant}

\renewcommand\leq{\leqslant}

\newcommand\be{\begin{eqnarray*}}
\newcommand\ee{\end{eqnarray*}}

\newcommand\K{\mathbb K}
\renewcommand\P{\mathbb P}
\renewcommand\L{\mathbb L}

\newcommand\calh{{\mathcal H}}

\newcommand\newop[2]{\def#1{\mathop{\rm #2}\nolimits}}
\newop\log{log}
\newop\ord{ord}
\newop\Gal{Gal}
\newop\SL{SL}
\newop\Bl{Bl}
\newop\mult{mult}
\newop\mass{mass}
\newop\div{div}
\newop\codim{codim}
\newop\sing{sing}
\newop\Zeroes{Zeroes}
\newop\Ass{Ass}
\newop\depth{depth}

\def\keywordname{{\bfseries Keywords}}%
\def\keywords#1{\par\addvspace\medskipamount{\rightskip=0pt plus1cm
\def\and{\ifhmode\unskip\nobreak\fi\ $\cdot$
}\noindent\keywordname\enspace\ignorespaces#1\par}}
\def\subclassname{{\bfseries Mathematics Subject Classification
(2010)}\enspace}
\def\subclass#1{\par\addvspace\medskipamount{\rightskip=0pt plus1cm
\def\and{\ifhmode\unskip\nobreak\fi\ $\cdot$
}\noindent\subclassname\ignorespaces#1\par}}

\newcommand\beginproof[1]{\trivlist\item[\hskip\labelsep{\em #1.}]}
\newcommand\proofof[1]{\beginproof{Proof of #1}}

\def\endproof{\hspace*{\fill}\endproofsymbol\endtrivlist}

\def\endproofsymbol{\frame{\rule[0pt]{0pt}{6pt}\rule[0pt]{6pt}{0pt}}}

\begin{document}

\author{ Mohammad Zaman Fashami, Hassan Haghighi, Tomasz Szemberg}
\title{On the fattening of ACM arrangements\\ of codimension 2 subspaces in $\mathbb{P}^N$}
\date{\today}
\maketitle
\thispagestyle{empty}

\begin{abstract}
   In the present note we study configurations of codimension $2$ flats in projective spaces
   and classify those with the smallest rate of growth of the initial sequence. Our work extends
   those of Bocci, Chiantini in $\P^2$ and Janssen in $\P^3$ to projective spaces of arbitrary dimension.
\subclass{13A15, 14N20, 14N05.}
\end{abstract}

%
%
\section{Introduction}
   A homogeneous ideal $I\subset R=\K[\P^N]$ in the ring of polynomials with coefficients in a field $\K$,
   decomposes
   as the direct sum of graded parts $I=\oplus_{t\geq 0}I_t$.
   For a nontrivial homogeneous ideal $I$ in $\K[\P^N]$, the \emph{initial degree}
   $\alpha(I)$ of $I$ is the least integer $t$ such that $I_t\neq 0$.

   For a positive integer $m$, the $m^{th}$ symbolic power $I^{(m)}$ of $I$ is defined as
   $$I^{(m)}=\bigcap_{P\in\Ass(I)}\left(I^mR_P\cap R\right),$$
   where $\Ass(I)$ is the set of associated primes of $I$ and the intersection takes
   place in the field of fractions of $R$.

   We define the \emph{initial sequence} of $I$ as the sequence of integers $\alpha_m=\alpha(I^{(m)})$.
   If $I$ is a radical ideal determined by the vanishing along a closed subscheme $Z\subset\P^N$,
   then the Nagata-Zariski Theorem \cite[Theorem 3.14]{EisenbudBook} provides a nice geometric
   interpretation of symbolic powers of $I$, namely $I^{(m)}$ is the ideal of polynomials
   vanishing to order at least $m$ along $Z$. This implies, in particular, that the initial
   sequence is strictly increasing.


   The study of the relationship between the sequence of initial degrees of symbolic powers of homogeneous
   ideals and the geometry of the underlying algebraic sets in projective spaces has been initiated by
   Bocci and Chiantini in \cite{BocChi11}. They proved that if $Z\subset\P^2$ is a finite set of points
   and $I=I(Z)$ is the vanishing ideal of $Z$, then the equality
   $$\alpha(I^{(2)})=\alpha(I)+1$$
   implies that either all   points in $Z$ are collinear or they form a star configuration (see Definition \ref{def:star}).

   This result has been considerably generalized in several directions.
   Dumnicki, Tutaj-Gasi\'nska and the third author studied higher symbolic powers of ideals supported
   on points in \cite{planar1} and \cite{planar2}.
   Natural analogies of the problem have been studied on $\P^1\times\P^1$ in \cite{P1xP1}
   and on Hirzebruch surfaces in general in \cite{DLS15}. Bauer and the third author
   proposed in \cite{BauSze15} the following conjecture for points in higher dimensional
   projective spaces and they proved it for $N=3$.
\begin{varthm*}{Conjecture}[Bauer, Szemberg]\label{conj:BS}
   Let $Z$ be a finite set of points in the projective space $\mathbb{P}^N$ and let $I$ be the radical ideal defining $Z$. If
\begin{equation*}
    d :=\alpha(I^{(N)})=\alpha(I) +N-1
\end{equation*}
   then either $\alpha(I)=1$ and the set $Z$ is contained in a single hyperplane
   or $Z$ is a star configuration of codimension $N$ associated to $d$ hyperplanes in $\P^N$.
\end{varthm*}
   In recent years many problems stated originally for points in projective spaces have been
   generalized to arrangements of flats, see e.g.
   \cite{GHV13, DHST14, MalSzp18, MalSzp17, DFST18} .
   In particular, Janssen in \cite{Jan15} generalized
   results of Bocci and Chiantini to configurations of lines in $\P^3$ defined by homogeneous Cohen-Macaulay ideals.
   Symbolic powers of codimension $2$ Cohen-Macaulay ideals have been studied recently in \cite{codim2}.
   Results of these two articles, especially Section 3 in \cite{Jan15}, have motivated our research presented here. Our main result is the following, (see
   Definition \ref{def:pseudo-star}, for the definition of a pseudo-star configuration).
\begin{theorem}[Main result]\label{thm:main}
   Let $\L$ be the union of a finite set of codimension $2$ projective subspaces in $\P^N$ and let $J$ be its vanishing ideal.
   If $J$ is Cohen-Macaulay and
   $$d:=\alpha(J^{(2)})=\alpha(J)+1,$$
   then $\L$ is either contained in a single hyperplane or it is a codimension $2$ pseudo-star configuration determined by $d$ hypersurfaces.
\end{theorem}
   Throughout this note we work over a field $\K$ of characteristic zero.
\section{Preliminaries}
   The term ''star configuration'' has been coined by Geramita. It is motivated by the observation that five general lines in $\P^2$
   resemble a pentagram. The objects defined below have appeared in recent years in various guises
   in algebraic geometry, commutative algebra and combinatorics, see
   \cite{GHM13} for a throughout account.
\begin{definition}[Star configuration]\label{def:star}
   Let $\calh=\left\{H_1,\ldots,H_s\right\}$ be a collection of $s\geq 1$ mutually distinct hyperplanes in $\P^N$
   defined by linear forms $\left\{h_1,\ldots,h_s\right\}$. We assume that the hyperplanes meet \emph{properly},
   i.e., the intersection of any $c$ of them is either empty or has codimension $c$, where $c$ is any integer
   in the range $1\leq c\leq\min\left\{s,N\right\}$. The union
   $$S(c,\calh)=\bigcup_{1\leq i_1< \ldots < i_c\leq s} H_{i_1}\cap\ldots\cap H_{i_c}$$
   is the \emph{codimension $c$ star configuration} associated to $\calh$.
   We have
   $$I(c,\calh)=I(S(c,\calh))=\bigcap_{1\leq i_1 < \ldots < i_c\leq s} \left(h_{i_1},\ldots,h_{i_c}\right).$$
   where $h_i$, $i = 1,\ldots, s $ are linear forms in $R$, defining the hyperplanes $H_i$.
\end{definition}
   The condition of meeting properly is satisfied by a collection of general hyperplanes.
   If the collection $\calh$ is clear from the context or irrelevant, we write
   $$S_N(c,s)$$
   to denote a codimension $c$ star configuration determined by $s$ hyperplanes in $\P^N$.

   For the purpose of this note, it is convenient to use the following terminology:
   an $r$--flat in a projective space is a linear subspace of (projective) dimension $r$.
   Thus a codimension $c$ star configuration determined by $s$ hyperplanes
   is the union of $\binom{s}{c}$ distinguished $(N-c)$--flats.

   The following notion is essential for our arguments.
\begin{definition}[Cohen-Macaulay]
   A noetherian local ring $(R,\mathfrak{m})$ is called \emph{Cohen-Macaulay}, if
   $$\depth_{\mathfrak{m}}R=\dim(R).$$
   A noetherian ring $R$ is Cohen-Macaulay (CM) if all of its local rings at prime ideals are Cohen-Macaulay.\\
  A closed
subscheme $Z\subset\P^N$ with defining ideal $I(Z)$, is called
  \emph{arithmetically Cohen-Macaulay (ACM for short)} if its coordinate ring
   $\K[\P^N]/I(Z)$ is CM.
\end{definition}
   By \cite[Proposition 2.9]{GHM13} every star configuration is ACM.

   The following feature of ACM subschemes makes them particularly suited for inductive
   arguments.
\begin{proposition}\label{prop:ACM and hyperplane section}
   Let $Z\subseteq \mathbb{P}^N$ be an ACM subscheme of dimension at least $1$,
   and let $H\subseteq \mathbb{P}^N$ be a general hyperplane. Then the intersection scheme $Z\cap H$ is ACM
   and
   $$\alpha(I(Z))=\alpha(I(Z\cap H)).$$
\end{proposition}
\begin{proof}
   A general hyperplane section of any curve is ACM, because all subschemes of dimension zero are ACM.
   For general hyperplane sections of higher dimensional subschemes see \cite[Theorem 1.3.3]{MiglioreBook}.
   The second claim follows from \cite[Corollary 1.3.8]{MiglioreBook} and some basic properties of postulation.
\end{proof}

   In \cite{Jan15} Janssen introduced the notion of a pseudo-star configuration for lines in $\P^3$.
   In the present note we extend this notion to higher dimensional flats in projective spaces
   of arbitrary dimension. To begin with, note that if $\calh=\left\{H_1,\ldots,H_s\right\}$
   is a collection of $s>N$ hyperplanes in $\P^N$, then the assumption that they intersect
   properly (see Definition \ref{def:star}) is equivalent to assuming that any $(N+1)$ of them
   have an empty intersection.
   For our purposes, we need to weaken this
condition, i.e.
\begin{definition}[Pseudo-star configuration]\label{def:pseudo-star}
   Let $\calh=\left\{H_1,\ldots,H_s\right\}$ be a collection of hyperplanes in $\P^N$ and let $1\leq c\leq N$
   be a fixed integer.
   We assume that the intersection of any $c+1$ of hyperplanes in $\calh$ has codimension $c+1$
   (equivalently: no $c+1$ hyperplanes in $\calh$ have the same intersection as any $c$ of them).
   The union
   $$P(c,\calh)=\bigcup_{1\leq i_1 < \ldots < i_c\leq s} H_{i_1}\cap\ldots\cap H_{i_c}$$
   is called the \emph{codimension $c$ pseudo-star configuration} determined by $\calh$.
\end{definition}
   If $\calh$ is clear from the context or irrelevant, we write $P_N(c,s)$ for a codimension
   $c$ pseudo-star configuration in $\P^N$ determined by $s$ hypersurfaces.
   Of course, any star configuration is a pseudo-star configuration. If $N=c=2$, then also the converse
   holds, i.e., any pseudo-star configuration of points in $\P^2$ is a star configuration. In general
   the two notions go apart, see Section \ref{sec:examples}.
  Moreover, being a pseudo-star configuration is stable under taking cones over the configuration.
\begin{remark}
   Let $I\subset\K[\P^N]$ be an ideal of a codimension $c$ pseudo-star configuration $P_N(c,s)$.
  Then the extension of the ideal $I$ to $\K[\P^{N+1}]$ defines a $P_{N+1}(c,s)$.
\end{remark}
   The construction known in the Liaison Theory as the \textit{Basic Double Linkage}
   (see \cite[chapter 4]{MiglioreBook}) was used in \cite{GHM13} to prove some basic properties of star configuration.
   These properties are also satisfied for pseudo-star configurations as the following proposition shows.
\begin{proposition}\label{prop:pseudo-star basic}
   Let $\calh =\lbrace H_1,\cdots , H_s\rbrace $ be a collection of mutually distinct hyperplanes in $\mathbb{P}^N$
   such that any $c+1$ of them intersect in a subspace of codimension $c+1$.
   Let $P(c,\calh)$ be the associated codimension $c$ pseudo-star configuration and let $I$ be its vanishing ideal.
   Then:
\begin{enumerate}
\item[1)] $\deg P(c,\calh)={s \choose c}$;
\item[2)] $P(c,\calh)$ is $ACM$;
\item[3)] $I^{(m)}$ is CM for all $1\leq m\leq c$;
\item[4)] $\alpha(I)=s-c+1$ and all minimal generators of $I$ occur in this degree.
\end{enumerate}
\end{proposition}
\begin{proof}
   According to the definition of a pseudo-star configuration 1) is obvious.
   Properties 2) and 4) were proved in \cite[Proposition 2.9]{GHM13} (see also \cite[Remark 2.13]{GHM13}).
   Symbolic powers of an ideal defining a pseudo-star configuration are Cohen-Macaulay by the first part of the proof of \cite[Theorem 3.2]{GHM13}.
   Note that Proposition 2.9 and Theorem 3.2 in \cite{GHM13} are stated for star configuration
   but the assumption that the hyperplanes meet properly can be relaxed to the assumption in the mentioned Proposition.
   Hence the proof of that Proposition, works for pseudo-stars too.
\end{proof}
\begin{remark}
   Since the ideal of every linear subspace in a projective space is a complete intersection, 
   by unmixedness theorem, we can describe the $m^{th}$ symbolic power of a pseudo-star configuration in a straightforward manner. In fact, let $\calh=\left\{H_1,\ldots,H_s\right\}$ be a pseudo-star configuration in $\P^N$, defined by the linear forms $h_1,\ldots,h_s \in R$. 
   Let $c \geq 1$ be a fixed integer. Then
   $$I=\bigcap_{1\leq i_1 < \ldots < i_c\leq s}(h_{i_1},\ldots,h_{i_s})$$
   is the defining ideal of codimension $c$ pseudo-star configuration associated to $\calh$. 
   Then by unmixedness theorem, for any positive integer $m$, one has
   $$I^{(m)}=\bigcap_{1\leq i_1 < \ldots < i_c\leq s}(h_{i_1},\ldots,h_{i_s})^m.$$
   \end{remark}
   We use this description, to compute the second symbolic powers of the ideals in the following section.
\section{Examples}\label{sec:examples}
In this section, we assume $R$ is $\K[\P^3]$.
\begin{example}[Star configurations]
   Let $\calh=\left\{H_1, H_2, H_3, H_4\right\}$ be a collection of hyperplanes in $\P^3$
   defined by the following linear forms
   $$h_1=x+2y+3z,\; h_2=x+y+w,\; h_3=x+z+w,\; h_4=y+z+w.$$
   These hyperplanes meet properly. Let $I$ be the ideal of the star configuration
   of lines $S(2,\calh)$ and let $J$ be the ideal of the associated star configuration of points $S(3,\calh)$.
   The minimal free resolution of $I, I^{(2)}, J$ and $J^{(2)}$ respectively are as follows 
\begin{align*}
0\rightarrow R^3(-4)\rightarrow R^4(-3)\rightarrow I \rightarrow 0,
\end{align*}
\begin{align*}
0\rightarrow R^4(-7)\rightarrow R(-4)\oplus R^4(-6)\rightarrow I^{(2)} \rightarrow 0,
\end{align*}
\begin{align*}
0\rightarrow R^3(-4)\rightarrow R^8(-3)\rightarrow R^6(-2)\rightarrow J\rightarrow 0,
\end{align*}
\begin{align*}
0\rightarrow R^6(-6)\rightarrow R^3(-4)\oplus R^{12}(-5) \rightarrow R^4(-3)\oplus R^6(-4)\rightarrow J^{(2)}\rightarrow 0.
\end{align*}
   We see immediately that $\alpha(I)=3$ and $\alpha(I^{(2)})=4$. Similarly for the ideal $J$ we have
   $\alpha(J)=2$ and $\alpha(J^{(2)})=3$.
\end{example}
Our next example, is a pseudo-star configuration in $\P^3$.
\begin{example}[A pseudo-star configuration]
   Let $\calh=\left\{H_1, H_2, H_3, H_4\right\}$ be a collection of hyperplanes in $\P^3$
   defined by the following linear forms
   $$h_1=x+5z,\; h_2=17x+19y,\; h_3=2x+3y+11z,\; h_4=13x+7z.$$
   These hyperplanes do not meet properly but the intersection of any three of them has codimension $3$
   and they all intersect in the point $P=(0:0:0:1)$. Let $I$ be the ideal of the pseudo-star configuration
   of lines $P(2,\calh)$.
   The minimal free resolutions of $I$ and  $I^{(2)}$ are
\begin{align*}
0\rightarrow R^3(-4)\rightarrow R^4(-3)\rightarrow I\rightarrow 0,
\end{align*}
\begin{align*}
0\rightarrow R^4(-7)\rightarrow R(-4)\oplus R^4(-6)\rightarrow I^{(2)}\rightarrow 0.
\end{align*}
   Now we have also $\alpha(I)=3$ and $\alpha(I^{(2)})=4$.

   Note that the codimension $3$ pseudo-star configuration defined by the ideal 
   $$J=(h_1, h_2, h_3)\cap (h_1, h_2, h_4)\cap (h_1, h_3, h_4)\cap (h_2, h_3, h_4)$$
   is now just a single point $\left\{ P\right\}$, i.e. its defining ideal in $\K[x,y,z,w]$ is $J=\langle x,y,z\rangle$.
\end{example}
\section{Proof of the Main Result}
   In the course of proving the main result of this note, the following lemma plays a crucial role. 
   In fact, it is a higher dimensional analogue of \cite[Proposition 2.10]{Jan15}.
\begin{lemma}\label{lem:H section}
   If a collection of $(N-2)$--planes in $\P^N$ with $N\geq 4$ is not contained in a hyperplane in $\P^N$
   (so that there are in particular at least $t\geq 2$ such planes), then its intersection with a \emph{general}
   hyperplane $H\subset \P^N$ is not contained in a hyperplane in $H$.
\end{lemma}
\begin{proof}
   It suffices to prove the statement for $t=2$. Let $U, V$ be $(N-2)$--planes in $\P^N$.
   By the dimension formula
   $$\dim\left\langle U,V\right\rangle=\dim U+\dim V-\dim (U\cap V)$$
   and by the assumption that $U$ and $V$ span $\P^N$, we have
   $$N=N-2+N-2-\dim(U\cap V),$$
   so that $\dim(U\cap V)=N-4$. Since $N\geq 4$ by assumption, the intersection $U\cap V$
   is non-empty. With the usual convention that the dimension of the empty set equals $-1$, we have
   for a general hyperplane $H$
   $$\dim(U\cap H)=\dim U-1=N-3,\;\; \dim(V\cap H)=\dim V-1=N-3,\; $$
   $$\mbox{and}\;\; \dim((U\cap V)\cap H)=\dim(U\cap V)-1=N-5.$$
   Hence
   $$\dim\left\langle U\cap H, V\cap H\right\rangle=N-3+N-3-N+5=N-1.$$
   This means that $(U\cap H)$ and $(V\cap H)$ span $H$ and we are done.
\end{proof}
\begin{remark}
   The above proof fails for two lines in $\P^3$. This is the reason that the argument in \cite{Jan15}
   is somewhat more involved. In fact, the proof of $N=3$ seems the most difficult case, contrary to what one
   might naively expect.
\end{remark}
   We are now in the position to prove our main result.
\proofof{Theorem \ref{thm:main}}
   Let $\L=L_1\cup\ldots\cup L_t$ be the union of $(N-2)$--flats in $\P^N$ such
   that the initial sequence $\alpha_m$ of the vanishing ideal $J$ of $\L$ satisfies
   $$d:=\alpha_2=\alpha_1+1$$
   and $J$ is CM.

   To prove our claim, we proceed by induction on $N$.
   For $N=2$ see \cite[Theorem 1.1]{BocChi11},
   and for $N=3$ see \cite[Theorem 2.13]{Jan15}. Moreover, if $t = 1$, then the claim is clear so we can assume $N\geq 4$ and $t\geq 2$.
   If $\L$ is contained in a hyperplane, then there is nothing to prove.
   So we assume that $\L$ spans the space $\P^N$. Let $H$ be a general hyperplane in $\P^N$.
   Then, the intersection $\L_H=\L\cap H$ can be represented as
   $$\L_H=(L_1\cap H)\cup\ldots\cup(L_t\cap H).$$
   Since $H$ is general, $\dim(L_i\cap H)=N-3$ for all $i=1,\ldots,t$.
   By Proposition \ref{prop:ACM and hyperplane section} the ideal $J_H$ of $\L_H$ is CM
   and its initial sequence $\beta_m=\alpha(J_H^{(m)})$ satisfies
   $$d=\beta_2=\beta_1+1.$$
   By the induction assumption, $\L_H$ is a codimension two pseudo-star configuration
   determined by hypersurfaces $F_1,\ldots,F_d$ in H. Indeed, $\L_H$ cannot be contained
   in a hyperplane since otherwise, by Lemma \ref{lem:H section}, $\L$ would be contained
   in a hyperplane. The hypersurface $F_1$ contains
   its intersections with the remaining $(d-1)$ hypersurfaces $F_2,\ldots,F_d$. These
   intersections are traces of some of the $(N-2)$-flats $L_1,\ldots,L_t$.
   There are exactly $\binom{d}{2}$ intersections among $F_i$'s by 1) in Proposition \ref{prop:pseudo-star basic}.
   There must be exactly as many traces so that $t=\binom{d}{2}$. Since the intersections
   of $F_1$ with $F_2,\ldots,F_d$ are by definition contained in $F_1$, a hyperplane in $H$,
   and they are on the other hand intersections of some $L_i$'s with $H$, the corresponding $L_i$'s
   must be themselves contained in a hyperplane, say $H_1$ in $\P^N$. Permuting the indices
   we obtain hypersurfaces $H_1,\ldots, H_d$ in $\P^N$ such that
   $$F_i=H_i\cap H\;\mbox{ for }\; i=1,\ldots,d.$$ Since every $(N-3)$--flat $(L_i\cap H)$
   is contained in exactly two of $F_i$'s (by the definition of a codimension two pseudo-star
   configuration), every $L_i$ must be contained in \emph{at least} two of $H_i$'s. But there are
   $\binom{d}{2}$ of $L_i$'s and at most that many pair intersections among the $H_i$'s.
   Hence every $L_i$ is contained in \emph{exactly} two of the $H_i$'s. This shows
   that $\L$ is the codimension two pseudo-star configuration determined by $H_1,\ldots,H_d$
   and we are done.
\endproof
   We complete the picture by showing that the converse statement holds for \emph{arbitrary}
   codimension two pseudo-star configurations.
\begin{theorem}\label{thm:complement}
   Let $\L$ be the union of $(N-2)$--flats $L_1,\ldots,L_t$ with the vanishing ideal $J$.
   If $\L$ is
   \begin{itemize}
   \item[a)] contained
   in a hyperplane, then the initial sequence of its vanishing ideal is 
   $$1,2,3,4,\ldots;$$
   \item[b)] a $P_N(2,s)$, then the initial sequence of its vanishing ideal is
   $$s-1,s,2s-1,2s,3s-1,3s,\ldots.$$
   \end{itemize}
\end{theorem}
\proof
   In case a) there is nothing to prove since the initial sequence is strictly increasing.\\
   In case b) we have, to begin with, $t=\binom{s}{2}$ for some $s$. Taking subsequent sections by general hyperplanes $H_1,\ldots,H_{N-2}$
   we arrive to a pseudo-star (hence a star) configuration of $\binom{s}{2}$ points in $\P^2$.
   In this case \cite[Proposition 3.2]{BocChi11} implies that $\alpha(J,H_1,\ldots,H_{N-2})=s-1$.
   This implies
   \begin{equation}\label{eq:1}
      \alpha(J)\geq s-1.
   \end{equation}
   On the other hand the union of $s$ hyperplanes in $\P^N$ vanishes to order two along $\L$, so that
   $\alpha(J^{(2)})\leq s$. Combining this with \eqref{eq:1} we obtain $\alpha(J)=s-1$ and $\alpha(J^{(2)})=s$.
   The argument for higher symbolic powers is similar and we leave the details to the reader.
\endproof
   In the view of our results, it is natural to conclude this note with the following challenge.
\begin{problem}
   Is there any codimension two pseudo-star configuration which is not ACM?
\end{problem}
\paragraph*{Acknowledgement.}
   This research has been carried out while the first author was a visiting fellow
   at the Department of Mathematics of the Pedagogical University of Cracow in the winter term 2017/18 and spring 2018.
   We thank Justyna Szpond for helpful comments.
   The research of the last named author was partially supported by
   National Science Centre, Poland, grant 2014/15/B/ST1/02197.

\bibliographystyle{abbrv}
\bibliography{master}


\bigskip \small

\bigskip

   Mohammad Zaman Fashami,
   Faculty of Mathematics, K. N. Toosi University of Technology, Tehran, Iran.

\nopagebreak
   \textit{E-mail address:} \texttt{zamanfashami65@yahoo.com}

\bigskip

   Hassan Haghighi,
   Faculty of Mathematics, K. N. Toosi University of Technology, Tehran, Iran.

\nopagebreak
   \textit{E-mail address:} \texttt{haghighi@kntu.ac.ir}

\bigskip

   Tomasz Szemberg,
   Department of Mathematics, Pedagogical University of Cracow,
   Podchor\c a\.zych 2,
   PL-30-084 Krak\'ow, Poland.

\nopagebreak
   \textit{E-mail address:} \texttt{tomasz.szemberg@gmail.com}

\bigskip


\end{document}